\documentclass[a4paper,twoside,11pt]{article}
\usepackage{a4wide}
\usepackage[utf8]{inputenc}
\usepackage[english]{babel} 
\usepackage{amsmath,amssymb,amsthm,amsfonts}
\usepackage{mathrsfs}
\usepackage{bbm}
\usepackage{graphicx}
\usepackage{hyperref}
\usepackage{fancyhdr}

\usepackage{enumitem}

\theoremstyle{plain}
\newtheorem{theorem}{Theorem}[section]
\newtheorem{lemma}[theorem]{Lemma}
\newtheorem{corollary}[theorem]{Corollary}
\newtheorem{proposition}[theorem]{Proposition}
\newtheorem{fact}[theorem]{Fact}

\theoremstyle{definition}

\theoremstyle{remark}

\numberwithin{equation}{section}

\def\N{\ensuremath{\mathbb{N}}}

\def\R{\ensuremath{\mathbb{R}}}

\def\E{\ensuremath{\mathbb{E}}}
\def\P{\ensuremath{\mathbb{P}}}

\def\Ind{\ensuremath{\mathbbm{1}}}

\def\to{\rightarrow}

\newcommand{\dd}{\mathrm{d}}

\author{\textsc{Pascal Maillard}\thanks{Department of Mathematics, The Weizmann Institute of Science, POB 26, Rehovot 76100, Israel, e-mail: \texttt{pascal DOT maillard AT weizmann DOT ac DOT il}}}
\title{A note on stable point processes occurring in branching Brownian motion}
\date{January 15, 2013}

\def\Mcal{\ensuremath{\mathcal{M}}}
\def\Ncal{\ensuremath{\mathcal{N}}}
\def\Mstar{\ensuremath{\Mcal^*}}
\def\Nstar{\ensuremath{\Ncal^*}}

\def\supp{\operatorname{supp}}

\DeclareMathOperator*{\esssup}{\operatorname{esssup}}

\begin{document}

\maketitle

{\leftskip=1.5truecm \rightskip=1.5truecm \baselineskip=15pt \small

\noindent{\bfseries Abstract.} We call a point process $Z$ on $\R$ \emph{exp-1-stable} if for every $\alpha,\beta\in\R$ with $e^\alpha+e^\beta=1$, $Z$ is equal in law to $T_\alpha Z+T_\beta Z'$, where $Z'$ is an independent copy of $Z$ and $T_x$ is the translation by $x$. Such processes appear in the study of the extremal particles of branching Brownian motion and branching random walk and several authors have proven in that setting the existence of a point process $D$ on $\R$ such that $Z$ is equal in law to $\sum_{i=1}^\infty T_{\xi_i} D_i$, where $(\xi_i)_{i\ge1}$ are the atoms of a Poisson process of intensity $e^{-x}\,\dd x$ on $\R$ and $(D_i)_{i\ge 1}$ are independent copies of $D$ and independent of $(\xi_i)_{i\ge1}$. In this note, we show how this decomposition follows from the classic \emph{LePage decomposition} of a (union)-stable point process. Moreover, we give a short proof of it in the general case of random measures on $\R$.

\bigskip


\bigskip


}

\section{Introduction}

Let $D$ be a point process on $\R$, $(D_i)_{i\ge 1}$ be independent copies of $D$ and $(\xi_i)_{i\ge1}$ be the atoms of a Poisson process of intensity $e^{-x}\,\dd x$ on $\R$ and independent of $(D_i)_{i\ge1}$. Suppose that the point process $Z$, defined as follows, exists.
\begin{equation}
\label{eq_definition_dppp}
Z = \sum_{i=1}^\infty T_{\xi_i} D_i
\end{equation}
It is then easy to see that for every $\alpha,\beta\in\R$ with $e^\alpha+e^\beta=1$, $Z$ is equal in law to $T_\alpha Z+T_\beta Z'$, where $Z'$ is an independent copy of $Z$ and $T_x$ is the translation by $x$. We call this property \emph{exp-1-stability} or \emph{exponential 1-stability} for a reason which will become clear later.

Processes of the form \eqref{eq_definition_dppp} arose during the study of the extremal particles in branching Brownian motion. Brunet and Derrida \cite[p.\ 18]{Brunet2011} asked the following question: Is it true that every exp-1-stable point process $Z$ admits the decomposition \eqref{eq_definition_dppp}? This question was answered in the affirmative by the author \cite{Maillard2011}, and independently in the special case appearing in branching Brownian motion by Arguin, Bovier, Kistler \cite{Arguin2011,Arguin2011a} and A\"{\i}d\'ekon, Berestycki, Brunet, Shi \cite{Aidekon2011c}. The decomposition \eqref{eq_definition_dppp} was also shown for the branching random walk by Madaule \cite{Madaule2011}, relying on the author's result. See also \cite{Kabluchko2011} for a related result concerning branching random walks. Note that the Poisson process with intensity  $e^{-x}\,\dd x$ is well-known in extreme value theory and describes the maxima of random variables which are independent and  identically distributed according to a law in the domain of attraction of the Gumbel distribution (see \cite[Corollary~4.19]{Resnick1987}). It therefore arises naturally here and in similar situations, for example in the theory of max-stable processes \cite{Kabluchko2009}.

Immediately after the article \cite{Maillard2011} was published on the arXiv, the author was informed by Ilya Molchanov that the representation \eqref{eq_definition_dppp} could be obtained from a classic result known as the \emph{LePage decomposition} of a \emph{stable point process}, which holds true in much more general settings.

The purpose of this note is two-fold: First, we want to outline how the theory of \emph{stability in convex cones} as developped by Davydov, Molchanov and Zuyev~\cite{Davydov2008} yields the above-mentioned LePage decomposition of stable point processes and with it the decomposition \eqref{eq_definition_dppp}. This is the content of Section~\ref{sec:cones}. Second, we give a succinct proof of the decomposition \eqref{eq_definition_dppp} for easy reference, a proof which uses more elementary methods than those of \cite{Davydov2008}. Furthermore, we give the extension of \eqref{eq_definition_dppp} to random measures, which cannot be directly obtained through the results of \cite{Davydov2008} (see Section~\ref{sec:cones}). The statements of the results (Theorem~\ref{th_random_measure} and Corollary~\ref{cor_PP}) and their proofs are the content of Section~\ref{sec:results}.

\subsection*{Branching Brownian motion}

In the remainder of this introduction, we outline the way exp-1-stable processes appear in branching Brownian motion (BBM). Define BBM as follows: Starting with one initial particle at the point $x$ of the real line, this particle performs Brownian motion until an exponentially distributed time of parameter $1/2$, at which it splits into two particles. Starting from the position of the split, both particles then repeat this process independently. 

We are interested in the point process formed by the right-most particles (draw the real line horizontally). It turns out that an important quantity is the so-called \emph{derivative martingale} $W_t = \sum_i (t-X_i(t)) \exp(X_i(t)-t)$, where we sum over all particles at time $t$ and denote the position of the $i$-th particle by $X_i(t)$. This martingale has an almost sure limit $W = \lim_t W_t > 0$ and it has been known since Bramson's \cite{Bramson1983} and Lalley and Sellke's \cite{Lalley1987} work that the position of the right-most particle, centred around $t-(3/2) \log t+\log W$, converges in law to a Gumbel distribution. By looking at a suitable Laplace transform \cite{Madaule2011}, one can strengthen this result to the whole point process $Z_t$ formed by the particles at time $t$. One obtains the existence of a point process $Z$ on $\R$, such that, starting from any configuration of finitely many particles, $T_{- t+(3/2) \log t-\log W}Z_t $  converges in law to $Z$ as $t\to\infty$.

Once the convergence of the point process is established, one now readily sees that the limiting process is exp-1-stable \cite{Brunet2011a,Madaule2011}: Take two BBMs and denote their derivative martingale limits by $W$ and $W'$, respectively. The union of both processes is then again a BBM with derivative martingale limit $W'' = W+W'$. Applying the before-mentioned convergence result to both BBMs as well as to their union, we get that for almost every realisation of $W$ and $W'$, $T_{\log(W+W')} Z$ is equal in law to $T_{\log W} Z + T_{\log W'} Z'$, where $Z$ and $Z'$ are iid and independent of $W$ and $W'$. Since $W$ and $W'$ can take any positive value (for example by varying the initial configurations), this yields the exp-1-stability of $Z$.

We emphasise that with this approach, one does not need to characterise the point process $Z$ directly, as it has been done before \cite{Arguin2011,Arguin2011a,Aidekon2011c}. This is helpful for models where such a direct characterisation would be complicated, for example for  branching random walks \cite{Madaule2011}.

\section{Stability in convex cones}
\label{sec:cones}

Let $Z$ be an exp-1-stable point process on $\R$. Define $Y$ to be the image (in the sense of measures) of $Z$ by the map $x\mapsto e^x$ (this was suggested by Ilya Molchanov). $Y$ is then a \emph{1-stable} point process on $(0,\infty)$, i.e.\ $Y$ is equal in law to $aY + bY'$, where $Y'$ is an independent copy of $Y$, $a,b\ge 0$ with $a+b=1$ and $aY$ is the image of $Y$ by the map $x\mapsto ax$. Note that if $Y$ is a simple point process (i.e.\ every atom has unit mass), then the set of its points is a random closed subset of $(0,\infty)$ and the stability property is then also known as the \emph{union-stability} for random closed sets (see e.g.\ \cite[Ch. 4.1]{Molchanov2005}).

Davydov, Molchanov and Zuyev~\cite{Davydov2008} have introduced a very general framework for studying stable distributions in \emph{convex cones}, where a convex cone $\mathbb K$ is a topological space equipped with two continuous operations: addition (i.e.\ a commutative and associative binary operation $+$ with neutral element $\mathbf e$) and multiplication by positive real numbers, the operations satisfying some associativity and distributivity conditions\footnote{One requires in particular that $a(x+y) = ax+ay$ for every $a>0$, $x,y\in\mathbb K$, but not that $(a+b)x = ax + bx$ for every $a,b>0$, $x\in\mathbb K$. }. Furthermore, $\mathbb K\backslash\{\mathbf e\}$ must be a complete separable metric space. For example, the space of compact subsets of $\R^d$ containing the origin is a convex cone, where the addition is the union of sets, the multiplication by $a>0$ is the image of the set by the map $x\mapsto ax$ and the topology is induced by the Hausdorff distance (see Example~8.11 in \cite{Davydov2008}). Furthermore, it is a \emph{pointed cone}, in the sense that there exists a unique \emph{origin} $\mathbf 0$, such that for each compact set $K\subset \R^d$, $aK\to \mathbf 0$ as $a\to 0$ (the origin is of course $\mathbf 0 = \{0\}$). The existence of 
the 
origin permits to define a \emph{norm} by $\|K\| = d(\mathbf 0,K)$, where $d$  is the Hausdorff distance. An example of a convex cone without origin (Example~8.23 in \cite{Davydov2008}) is the space of (positive) Radon measures on $\R^d\backslash\{0\}$ equipped with the topology of vague convergence, the usual addition of measures and multiplication by $a>0$ being defined as the image of the measure by the map $x\mapsto ax$, as above. 

A random variable $Y$ with values in $\mathbb K$ is called \emph{$\alpha$-stable}, $\alpha>0$, if $a^{1/\alpha}Y + b^{1/\alpha}Y'$ is equal in law to $(a+b)^{1/\alpha}Y$ for every $a,b>0$, where $Y'$ is an independent copy of $Y$. With the theory of Laplace transforms and infinitely divisible distributions on semigroups (the main reference to this subject is \cite{Berg1984}), the authors of \cite{Davydov2008} show that to every $\alpha$-stable random variable $Y$ there uniquely corresponds a \emph{L\'evy measure} $\Lambda$ on a certain second dual of $\mathbb K$ which is \emph{homogeneous of order $\alpha$}, i.e.\ $\Lambda(aB) = a^\alpha\Lambda(B)$ for any Borel set $B$. Since $\Lambda$ is \emph{a priori} only defined on this second dual of $\mathbb K$, a considerable part of the work in \cite{Davydov2008} is to give conditions under which $\Lambda$ is supported by $\mathbb K$ itself. Moreover, and this is their most important result, under some additional conditions, $Y$ can be represented by its \emph{LePage series}, i.e.\ the sum over the atoms of a Poisson process on $\mathbb K$ with intensity measure $\Lambda$.


Assuming that all the above conditions are verified, one can now disintegrate the homogeneous L\'evy measure $\Lambda$ into a radial and an angular component, such that $\Lambda = cr^{-\alpha-1}\dd r\times \sigma$ for $c>0$ and some measure $\sigma$ on the unit sphere $\mathbb S=\{x\in\mathbb K: \|x\|=1\}$. This is also called the \emph{spectral decomposition} and $\sigma$ is called the \emph{spectral measure}. If $\sigma$ has unit mass, then the LePage series can be written as
\begin{equation}
\label{eq_lepage}
Y = \sum_i \xi_i X_i,
\end{equation}
where $\xi_1,\xi_2,\ldots$ are the atoms of a Poisson process of intensity $cr^{-\alpha-1}\dd r$ and $X_1,X_2,\ldots$ are iid with law $\sigma$, independent of the $\xi_i$.

This allows us to prove the decomposition \eqref{eq_definition_dppp} for a simple exp-1-stable point process~$Z$: If $Y$ is  the point process obtained from $Z$ through the exponential transformation from the beginning of this section then $\supp Y \cup \{0\}$ is a random compact subset of $\R$ containing the origin, assuming the process $Z$ almost surely has only finitely many points in $\R_+$ (we will prove this simple fact in Lemma~\ref{lem_density} below). Hence, it is a random element of the cone from the first example given above. This cone satisfies the conditions required in \cite{Davydov2008}, such that the results there can be applied to yield the LePage decomposition \eqref{eq_lepage} of $Y$. This immediately implies the  decomposition \eqref{eq_definition_dppp} for $Z$.

If $Z$ is a general random measure on $\R$, the same exponential transformation can be applied, such that $Y$ becomes a 1-stable random measure on $(0,\infty)$, i.e.\ an element of the cone from the second example above. Unfortunately, this cone does not satisfy the conditions in  \cite{Davydov2008}, such that their results cannot be used directly\footnote{In particular, the theorems in \cite{Davydov2008} require that the cone be pointed and that the stable random elements have no Gaussian component, both conditions being violated by the cone of random measures (see the remark after Fact~\ref{fact_idrm} for the second condition). Note however that although the cone does not have an origin, it is still possible to define a ``norm'' on the subspace of random measures which assign finite mass to $[1,\infty)$, see the definition of the map $M$ in Section~\ref{sec:definitions}.}, although their very general methods could probably be applied in this setting as well.


\section{A succinct proof of the decomposition (\ref{eq_definition_dppp})}
\label{sec:results}

As mentioned in the introduction, we will give here a short proof of the decomposition \eqref{eq_definition_dppp} and its extension to random measures, effectively yielding a LePage decomposition for stable random measures on $(0,\infty)$. Instead of applying the general methods of harmonic analysis on semigroups used in \cite{Davydov2008}, we will rely on the much more elementary treatment of Kallenberg \cite{Kal1983} on random measures. We hope that our proof will be more accessible to probabilists who are not familiar with the methods used in \cite{Davydov2008}. Note that it can be easily generalised to give a LePage decomposition for stable random measures on $\R^d\backslash\{0\}$ or more general spaces. However, for simplicity and because of its interest in applications, we will stick to the one-dimensional setting. For the same reasons, we will also keep the notion of exp-stability instead of the usual stability.


\subsection{Definitions and notation}
\label{sec:definitions}

We denote by $\Mcal$ the space of (positive) Radon measures on $\R$. Note that $\mu\in\Mcal$ if and only if $\mu$ assigns finite mass to every bounded Borel set in $\R$. We further denote by $\Ncal$ the subspace of counting (i.e.\ integer-valued) measures. It is known (see e.g.\ \cite{DV2003}, p.\ 403ff) that there exists a metric $d$ on $\Mcal$ which induces the vague topology and under which $(\Mcal,d)$ is complete and separable. We further set $\Mstar = \Mcal \backslash\{0\}$ (where $0$ denotes the null measure), which is an open subset and hence a complete separable metric space when endowed with the metric $d^*(\mu,\nu) = d(\mu,\nu) + |d(\mu,0)^{-1} - d(\nu,0)^{-1}|$, equivalent to $d$ on $\Mstar$ (\cite{BouGT2}, IX.6.1, Proposition~2). The spaces $\Ncal$ and $\Nstar = \Ncal\backslash\{0\}$ are closed subsets of $\Mcal$ and $\Mstar$, and therefore complete separable metric spaces as well (\cite{BouGT2}, IX.6.1, Proposition~1).

For every $x\in\R$, we define the translation operator $T_x:\Mcal\to\Mcal$, by $(T_x \mu)(A) = \mu(A-x)$ for every Borel set $A\subset\R$. Furthermore, we define the measurable map $M:\Mcal\to\R\cup\{+\infty\}$ by
\[
M(\mu) = \inf\{x\in\R: \mu((x,\infty)) < 1\wedge (\mu(\R)/2)\},
\] 
where we use the notation $x\wedge y = \min(x,y)$ and define $\inf\emptyset = \infty$ (in particular,  $M(0) = +\infty$). Note that for $\mu\in\Mstar$, we have $M(\mu)<\infty$ if and only if $\mu(\R_+)<\infty$. If furthermore $\mu\in\Nstar$, then $M(\mu)$ is the position of the rightmost atom of $\mu$, i.e.\ $M(\mu) = \esssup \mu$. It is easy to show that the maps $(x,\mu)\mapsto T_x\mu$ and $M$ are continuous, hence measurable.

A \emph{random measure} $Z$ on $\R$ is a random variable taking values in $\Mcal$. If $Z$ takes values in $\Ncal$, we also call $Z$ a \emph{point process}. Let $\mathscr F$ denote the set of non-negative measurable functions $f:\R\to\R_+=[0,\infty)$. For every $f\in\mathscr F$, we define the \emph{cumulant}
\[K(f) = K_Z(f) = -\log \E\left[\exp(-\langle Z,f\rangle )\right]\in[0,\infty],\]
where $\langle \mu,f\rangle  = \int_\R f(x)\mu(\dd x)$. The cumulant uniquely characterises $Z$ (\cite{DV2003}, p.\ 161).

We say that $Z$ is \emph{exp-1-stable} or simply \emph{exp-stable} if for every $\alpha,\beta\in\R$ with $e^\alpha+e^\beta=1$, $Z$ is equal in law to $T_\alpha Z+T_\beta Z'$, where $Z'$ is an independent copy of $Z$.

The following theorem and its corollary are precise statements of the decomposition \eqref{eq_definition_dppp} and form the main results of this paper.

\begin{theorem}
\label{th_random_measure}
A function $K:\mathscr F\to\R_+$ is the cumulant of an exp-stable random measure on $\R$ if and only if for every $f\in\mathscr F$,
\begin{equation}
\label{eq_K_exp-stable}
K(f) = c\int_\R e^{-x}f(x)\,\dd x + \int_\R e^{-x} \int_{\Mstar} [1- \exp(-\langle \mu,f \rangle)] T_x\Delta(\dd \mu)\,\dd x,
\end{equation}
for some constant $c\ge0$ and some measure $\Delta$ on $\Mstar$, such that for every bounded Borel set $A\subset\R$,
\begin{equation}
\label{eq_condition_Delta}
\int_\R e^x\int_0^\infty (1\wedge y) \Delta(\mu(A+x)\in\dd y)\,\dd x < \infty.
\end{equation} Moreover, $\Delta$ can be chosen such that $\Delta(M(\mu)\ne 0) = 0$, and as such, it is unique.
\end{theorem} 


\begin{corollary}
\label{cor_PP}
A point process $Z$ on $\R$ is exp-stable if and only if it has the representation \eqref{eq_definition_dppp} for some point process $D$ on $\R$ satisfying 
\begin{equation}
\label{eq_condition_D_PP}
\int_0^{\infty} \P(D(A+x) > 0) e^x\,\dd x < \infty.
\end{equation}
Moreover, if the above holds, then there exists a unique pair $(m,D)$ with $m\in\R\cup\{+\infty\}$ and $D$ a point process on $\R$ such that $\P(M(D) = m) = 1$ and \eqref{eq_definition_dppp} and \eqref{eq_condition_D_PP} are satisfied.
\end{corollary}

\subsection{Infinitely divisible random measures}
\label{sec_infinite_divisibility}
Our proof of Theorem~\ref{th_random_measure} is based on the theory of infinitely divisible random measures as exposed in Kallenberg \cite{Kal1983}. A random measure $Z$ is said to be \emph{infinitely divisible} if for every $n\in\N$ there exist iid random measures $Z^{(1)},\ldots,Z^{(n)}$ such that $Z$ is equal in law to $Z^{(1)} + \cdots + Z^{(n)}$. It is said to be infinitely divisible as a point process if $Z^{(1)}$ can be chosen to be a point process. Note that a (deterministic) counting measure is infinitely divisible as a random measure but not as a point process.

The main result about infinitely divisible random measures is the following (see \cite{Kal1983}, Theorem 6.1 or \cite{DV2008}, Proposition 10.2.IX, however, note the error in the theorem statement of the latter reference: $F_1$ may be infinite as it is defined).

\begin{fact}
\label{fact_idrm}
A random measure $Z$ with cumulant $K(f)$ is infinitely divisible if and only if
\[K(f) = \langle\lambda,f\rangle + \int_{\Mstar} [1- \exp(-\langle \mu,f \rangle)] \Lambda(\dd \mu),\]
where $\lambda\in\Mcal$ and $\Lambda$ is a measure on $\Mstar$ satisfying
\begin{equation}
\label{eq_Lambda}
\int_0^\infty (1\wedge x) \Lambda(\mu(A) \in \dd x) < \infty,
\end{equation}
for every bounded Borel set $A\subset\R$.
\end{fact}
The probabilistic interpretation (\cite{Kal1983}, Lemma 6.5) of this fact is that $Z$ is the superposition of the non-random measure $\lambda$ and of the atoms of a Poisson process on $\Mstar$ with intensity $\Lambda$. In the general framework of infinitely divisible distributions on semigroups used in \cite{Davydov2008} the measures $\lambda$ and $\Lambda$ are called the \emph{Gaussian component} and the \emph{L\'evy measure}, respectively. Fact~\ref{fact_idpp} has the following analogue in the case of point processes (\cite{DV2008}, Proposition 10.2.V), where the measure $\Lambda$ is also called the \emph{KLM measure}. Note that the Gaussian component disappears.

\begin{fact}
\label{fact_idpp}
A point process $Z$ is infinitely divisible as a point process if and only if $\lambda = 0$ and $\Lambda$ is concentrated on $\Nstar$, where $\lambda$ and $\Lambda$ are the measures from Fact \ref{fact_idrm}. Then, \eqref{eq_Lambda} is equivalent to $\Lambda(\mu(A) > 0) < \infty$ for every bounded Borel set $A\subset\R$.
\end{fact}

In particular, the L\'evy/KLM measure of a Poisson process on $\R$ with intensity measure $\nu(\dd x)$ is the image of $\nu$ by the map $x\mapsto \delta_x$.

\subsection{Proof of Theorem~\ref{th_random_measure}}
\label{sec_proof}
We can now prove Theorem~\ref{th_random_measure} and Corollary~\ref{cor_PP}. For the ``if'' part, we note that \eqref{eq_condition_Delta} implies \eqref{eq_Lambda} for the measure $\Lambda = \int e^{-x}T_x\Delta\,\dd x$, such that the process with cumulant given by \eqref{eq_K_exp-stable} exists. The exp-stability is readily verified. Further note that for point processes the condition \eqref{eq_condition_D_PP} is equivalent to \eqref{eq_condition_Delta}.

It remains to prove the ``only if'' parts. Let $Z$ be an exp-stable random measure. Then, for $\alpha,\beta\in\R$, such that $e^\alpha + e^\beta = 1$, we have
\[\begin{split}
K(f) = -\log \E[\exp(-\langle Z,f\rangle )] &= -\log \E[\exp(-\langle T_\alpha Z,f\rangle )] - \log \E[\exp(-\langle T_\beta Z,f\rangle )]\\
&= K(f(\cdot + \alpha)) + K(f(\cdot + \beta)).
\end{split}\]
Setting $\varphi(x) = K(f(\cdot + \log x))$ for $x \in \R_+$ (with $\varphi(0) = 0$) and replacing $f$ by $f(\cdot + \log x)$ in the above equation, we get $\varphi(x) = \varphi(xe^\alpha) + \varphi(xe^\beta)$ for all $x\in\R_+$, or $\varphi(x)+\varphi(y) = \varphi(x+y)$ for all $x,y\in\R_+$. This is the famous Cauchy functional equation and since $\varphi$ is by definition non-negative on $\R_+$, it is known and easy to show \cite{Darboux1880} that $\varphi(x) = \varphi(1) x$ for all $x\in\R_+$. As a consequence, we obtain the following corollary:
\begin{corollary}
\label{cor_exp} $K(f(\cdot+x)) = e^x K(f)$ for all $x\in\R$.
\end{corollary}

Furthermore, it is easy to show that exp-stability implies infinite divisibility. We then have the following lemma.

\begin{lemma}
\label{lem_density}
Let $\lambda,\Lambda$ be the measures corresponding to $Z$ by Fact~\ref{fact_idrm}. 
\begin{enumerate}[nolistsep]
\item There exists a constant $c\ge0$, such that $\lambda = ce^{-x}\,\dd x$.
\item For every $x\in\R$, we have $T_x\Lambda = e^x\Lambda$.
\item For $\Lambda$-almost every $\mu$, we have $\mu(\R_+) < \infty$.
\end{enumerate}
\end{lemma}
\begin{proof}
The measures $T_x\lambda$, $T_x\Lambda$ are the measures corresponding to the infinitely divisible random measure $T_xZ$ by Fact \ref{fact_idrm}. But by Corollary \ref{cor_exp}, the measures $e^x\lambda$ and $e^x \Lambda$ correspond to $T_xZ$, as well. Since these measures are unique, we have $T_x\lambda = e^x\lambda$ and $T_x\Lambda = e^x\Lambda$. The second statement follows immediately. For the first statement, note that $c_1 = \lambda([0,1)) < \infty$, since $[0,1)$ is a bounded set. It follows that \[\lambda([0,\infty)) = \sum_{n\ge0} \lambda([n,n+1)) = \sum_{n\ge0} c_1e^{-n} = \frac{c_1e}{e-1} =: c,\]
hence $\lambda([x,\infty)) = ce^{-x}$ for every $x\in\R$. The first statement of the lemma follows. For the third statement, let 
  $I_n = [n,n+1)$ and $I = [0,1)$. By \eqref{eq_Lambda}, we have \[\int_0^1\Lambda(\mu(I) > x) \,\dd x = \int_0^1 x\Lambda(\mu(I)\in \dd x)  <\infty.\] By monotonicity, the first integral is greater than or equal to $x \Lambda(\mu(I) > x)$ for every $x\in[0,1]$, hence $\Lambda(\mu(I) > x) \le C/x$ for some constant $0\le C < \infty$. By the second statement, it follows that
\[
\Lambda(\mu(I_n) > e^{-n/2}) = e^{-n} \Lambda(\mu(I) > e^{-n/2}) \le C e^{-n/2},
\] 
for every $n\in\N$. Hence, $\sum_{n\in\N} \Lambda\left(\mu(I_n) > e^{-n/2}\right) < \infty$. By the Borel-Cantelli lemma,
\[\Lambda\left(\limsup_{n\to\infty} \left\{\mu(I_n) > e^{-n/2}\right\}\right) = 0,\] 
which implies the third statement.
\end{proof}

\begin{lemma}
\label{lem_decomp}
The measure $\Lambda$ admits the decomposition $\Lambda = \int e^{-x}T_x\Delta\,\dd x$, where $\Delta$ is a unique measure on $\mathcal M^*$ with $\Delta(M(\mu)\ne 0) = 0$.
\end{lemma}

\begin{proof}
  We follow the proof of Proposition~4.2 in \cite{Rosinski1990}. Set $\Mstar_0 := \{\mu\in\Mstar:M(\mu)=0\}$ and $\Mstar_\R := \{\mu\in\Mstar:M(\mu)<\infty\}$, which are measurable subspaces of the complete separable metric space $\Mstar$ and therefore Borel spaces \cite[Theorem~A1.6]{Kallenberg1997}. By the continuity of $(x,\mu)\mapsto T_x\mu$, the map $\phi: \Mstar_\R\to \Mstar_0\times \R$ defined by $\phi(\mu) = (T_{-M(\mu)}\mu,M(\mu))$ is a Borel isomorphism, i.e.\ it is bijective and $\phi$ and $\phi^{-1}$ are measurable. The translation operator $T_x$ acts on $\Mstar_0\times \R$ by $T_x(\mu,m) = (\mu,m+x)$.

Now note that $\Lambda$ is supported on $\Mstar_\R$ by the third part of Lemma~\ref{lem_density}. Denote by $\Lambda^\phi$ the image of $\Lambda$ by the map $\phi$ and set $A_n = \{\mu\in\Mstar_0:\mu([-2n,2n]) \ge 1/n\}$. Then $\Lambda^\phi(A_n\times[-n,n])<\infty$ for every $n\in\N$ by \eqref{eq_Lambda}. By the theorem on the existence of conditional probability distributions (see e.g.\ \cite{Kallenberg1997}, Theorems~5.3 and 5.4) there exists then a measure $\Delta_0$ on $\Mstar_0$ with $\Delta_0(A_n)<\infty$ for every $n\in\N$ and a measurable kernel $K(\mu,\dd m)$, with $K(\mu,[-n,n])<\infty$ for every $n\in\N$, such that 
  \begin{equation*}
\Lambda^\phi(\dd \mu,\dd m) = \int_{\Mstar_0} \Delta_0(\dd\mu)K(\mu,\dd m).
  \end{equation*}
Moreover, we can assume in the above construction that $K(\mu,[0,1]) = 1$ for every $\mu\in\Mstar_0$ and $n\in\N$, and with this normalization, $\Delta_0$ is unique. By Lemma~\ref{lem_density}, we now have $T_xK(\mu,\dd m) = e^x K(\mu,\dd m)$ for every $x\in\R$ and $\mu\in\Mstar_0$. As in the proof of the first statement of Lemma~\ref{lem_density}, we then conclude that $K(\mu,\dd m) = c(\mu)e^{-m}\,\dd m$ for some constant $c(\mu)\ge 0$, and by the above normalization, $c(\mu)\equiv c := e/(e-1)$. Setting $\Delta(\dd m) = c\Delta_0(\dd m)$ then gives
\begin{equation*}
  \Lambda^\phi(\dd \mu,\dd m) = \int_{\Mstar_0} \Delta(\dd\mu)e^{-m}\,\dd m.
\end{equation*}
Mapping $\Lambda^\phi$ back to $\Mstar_\R$ by the map $\phi^{-1}$ finishes the proof.
\end{proof}

The ``only if'' part of Theorem~\ref{th_random_measure} now follows from the previous lemmas and Fact~\ref{fact_idrm}. As for the proof of Corollary~\ref{cor_PP}, if $Z$ is a point process, then Fact~\ref{fact_idpp} implies that $\lambda = 0$ and that $\Lambda$ is concentrated on $\Nstar$, hence $\Delta$ as well. Equation \eqref{eq_condition_Delta} then implies that $\Delta(\mu(A)>0) < \infty$ for any bounded Borel set $A\subset\R$. In particular, this holds for $A=\{0\}$. But by Lemma~\ref{lem_decomp},  $\Delta$ is concentrated on $\Nstar_0=\{\mu\in\Nstar:M(\mu)=0\}$ and is therefore a finite measure, since $\mu\in\Nstar_0$ implies $\mu(\{0\}) > 0$.

Now, if $\P(Z\ne 0)>0$ (the other case is trivial), then $\Delta(\Nstar_0)>0$ and we set $m=\log \Delta(\Nstar_0)$. The measure $\Delta' = e^{-m}T_m\Delta$ is then a probability measure and $\Lambda = \int e^{-x}T_x\Delta'\,\dd x$. Furthermore, $Z$ satisfies \eqref{eq_definition_dppp}, where $D$ follows the law $\Delta'$. Uniqueness of the pair $(m,D)$ follows from Lemma~\ref{lem_decomp}. This finishes the proof of Corollary~\ref{cor_PP}.

\subsection{Finiteness of the intensity}
\label{sec_dppp}
If $Z$ is an exp-stable point process and has finite intensity (i.e.\ $E[Z(A)] < \infty$ for every bounded Borel set $A\subset \R$), then it is easy to show that the intensity is proportional to $e^{-x}\,\dd x$. However, in the process which occurs in the extremal particles of branching Brownian motion or branching random walk, the intensity of the point process $D$ grows with $|x|e^{|x|}$, as $x\to -\infty$ \cite[Section~4.3]{Brunet2011a}. The following simple result shows that in these cases, $Z$ does not have finite intensity.
\begin{proposition}
\label{prop_dppp_exp}
Let $Z$ be an exp-stable point process on $\R$ and let $D$ be the point process from Corollary~\ref{cor_PP}. Then $Z$ has finite intensity if and only if $\E[\langle D,e^x\rangle ] < \infty$.
\end{proposition}

\begin{proof}
By the Fubini--Tonelli theorem,
\[E[Z(A)] = \E\left[\sum_{i\in\N} \E[T_{\xi_i}D(A)\,|\,\xi]\right]= \int_\R \E[D(A-y) e^{-y}]\, \dd y = \E\left[\int_\R D(A-y) e^{-y}\, \dd y\right],\]
for every bounded Borel set $A\subset \R$. Again by the Fubini--Tonelli theorem we have
\[\int_\R D(A-y) e^{-y}\, \dd y = \int_\R \int_\R \Ind_{A-y}(x)e^{-y}\, \dd y\, D(\dd x) = \langle D,\int_\R \Ind_{A-y}(\cdot)e^{-y}\, \dd y \rangle.\]
For $x\in\R$, $x\in A-y$ implies $y \in [\min A - x,\max A - x]$. Since $e^{-y}$ is decreasing, we therefore have
\[|A|e^{-\max A}e^x \le \int_\R \Ind_{A-y}(x)e^{-y}\, \dd y \le |A|e^{-\min A}e^x,\] where $|A|$ denotes the Lebesgue measure of $A$. We conclude that $E[Z(A)] < \infty$ if and only if $\E[\langle D,e^x\rangle ] < \infty.$
\end{proof}

\providecommand{\bysame}{\leavevmode\hbox to3em{\hrulefill}\thinspace}
\providecommand{\MR}{\relax\ifhmode\unskip\space\fi MR }
\providecommand{\MRhref}[2]{%
  \href{http://www.ams.org/mathscinet-getitem?mr=#1}{#2}
}
\providecommand{\href}[2]{#2}

\section*{Acknowledgements}
I thank two anonymous referees for having spotted several typographical errors in the manu\-script and for having requested more details and explanations, which greatly benefitted the presentation. One referee pointed out the reference \cite{Kabluchko2009}.

\end{document}